\documentclass[12pt]{amsart}
\usepackage{amssymb}
\usepackage{amsthm}
\usepackage{amsmath}
\usepackage{graphicx}

\newtheorem{theorem}{Theorem}[section]

\newtheorem{definition}[theorem]{Definition}
\newtheorem{lemma}[theorem]{Lemma}

\newcommand{\Fq}{\ensuremath{\mathbb{F}_q}}
\newcommand{\Fd}{\ensuremath{\mathbb{F}_2}}
\newcommand{\PG}{{\rm PG}}
\newcommand{\B}{\ensuremath{\mathcal{B}}}

\newcommand{\I}{\ensuremath{\mathcal{I}}}
\newcommand{\Sm}{\ensuremath{\mathcal{M}}}

\begin{document}

\title[Tactical decompositions of designs over finite fields]
{Tactical decompositions of designs over finite fields}

\author{Anamari Naki\'{c}}
\address{University of Zagreb, Faculty of electrical engineering
and computing, Department of applied mathematics, Unska~3,
\hbox{HR-10000 Zagreb,} Croatia} \email{anamari.nakic@fer.hr}

\author{Mario Osvin Pav\v{c}evi\'{c}}
\address{University of Zagreb, Faculty of electrical engineering
and computing, Department of applied mathematics, Unska~3,
\hbox{HR-10000 Zagreb,} Croatia} \email{mario.pavcevic@fer.hr}

\subjclass{05B05}

\keywords{design over finite field; tactical decomposition}

\date{September, 2013}

\begin{abstract}
An automorphism group of an incidence structure $\I$ induces a tactical decomposition on $\I$.
It is well known that tactical decompositions of $t$-designs satisfy certain necessary conditions which can be expressed as equations in terms of the coefficients of tactical decomposition matrices. In this article we present results obtained for tactical decompositions of $q$-analogs of $t$-designs, more precisely, of $2$-$(v,k,\lambda_2;q)$ designs. We show that coefficients of tactical decomposition matrices of a design over finite field satisfy an equation system analog to the one known for block designs. Furthermore, taking into consideration specific properties of designs over the binary field $\Fd$, we obtain an additional system of inequations for these coefficients in that case. 

\end{abstract}

\maketitle

\section{Introduction and preliminary results}
Let $\Fq$ be the finite field of order $q$ and $\Fq^v$ the
vector space of dimension $v$ over the finite field $\Fq$.
An $r$-space is a subspace of $\Fq^v$ of dimension $r$.
The number of $r$-spaces of $\Fq^v$ is
$$
{v \brack r}_q = \frac{(q^v-1)\cdots(q^{v-r+1}-1)}{(q^r-1)\cdots(q-1)}.
$$
The number of $r$-spaces containing a fixed $s$-space, $s \le r$, is
$$
{v-s \brack r-s }_q.
$$
For every two subspaces $U$ and $V$, \textit{the dimension formula} holds:
$$
\dim\,(\langle U, V \rangle) = \dim \, U + \dim \, V - \dim\,(U \cap V).
$$
Designs over finite fields were first introduced in the 1970's, see~\cite{CI74},\cite{CII74},\cite{D76}.
First nontrivial designs over finite fields which are not spreads
were constructed in~\cite{ST87}.
\begin{definition}
A finite set $\B$ is called design over finite field
with parameters $t$-$(v,k,\lambda_t)$ if the following properties hold:
\begin{enumerate}
\item elements of $\B$ are $k$-spaces of the vector space $\Fq^v$ called blocks,
\item every $t$-space of $\Fq^v$
is contained in $\lambda_t$ blocks.
\end{enumerate}
\end{definition}
Designs over finite fields are also often called $q$-analogs of $t$-designs,
or shorter $q$-designs. A $t$-$(v,k,\lambda_t)$ \emph{design} is a finite incidence structure
$ ({\mathcal P}, {\mathcal B} ) $, where $\mathcal{P}$ is a
set of $v$ elements called \emph{points}, and $\mathcal{B}$ is a
multiset of nonempty $k$-subsets of $\mathcal{P}$ called \emph{blocks}
such that every set of $t$ distinct points is
contained in exactly $\lambda_t$ blocks. When parameters are not important,
$t$-$(v,k,\lambda_t)$ designs are shorter called $t$-designs. When $t=2$, designs are called \emph{block designs}.
Designs over finite fields are closely related
to $t$-designs. Every design $\B$ with parameters $2$-$(v, k, \lambda_2; q)$
gives a block design with parameters
$2$-$({v \brack 1}_q, {k \brack 1}_q , \lambda_2)$, where points are
identified with $1$-spaces of $\Fq^v$ and each
block is identified with the set of $1$-spaces it contains.
The inverse statement is not valid. For example, there are
block designs with parameters $2$-$(15,7,3)$
which cannot be constructed from the associated $2$-$(4,3,3;2)$ design.

If $\B$ is a design with parameters $t$-$(v, k, \lambda_t; q)$, then
$\B$ is a design with parameters $s$-$(v,k,\lambda_s; q)$,
$0 \le s \le t$, where
$$
\lambda_s = \lambda_t \frac{{v-s \brack t-s}_q}{{k-s \brack t-s}_q}.
$$
The number of blocks in $\B$ equals
$$
|\B| = \lambda_t \frac{{v \brack t}_q}{ {k \brack t}_q }.
$$
\emph{Automorphism} of $\B$
is a linear operator $\Phi \in GL_v(q)$ such that
$ \Phi\B = \B$.
The set $Aut \, \B$ of all automorphisms of $\B$
is a subgroup of the general linear group $GL_v(q)$,
called \emph{full automorphism group} of $\B$.
Any subgroup of $Aut \, \B$ is an automorphism group of $\B$.

\section{Tactical decompositions of designs over finite fields}
The idea of considering tactical decompositions of block designs
was first introduced by Dembowski~\cite{D68}. Equations
for coefficients of tactical decomposition matrices
for block designs are well known~\cite{JT85} and they were
used for constructions of many examples of block designs (listed
in~\cite{MR07}). These equations were generalized for any $t \ge 1$ in~\cite{KNP13}.
In this article we introduce tactical decompositions of
designs over finite fields for $t = 2$. We show that coefficients of
tactical decomposition matrices satisfy an equation system analog
to the one known for block designs. Furthermore, taking into consideration specific properties of designs over the binary field $\Fd$, we obtain an additional system of inequations for coefficients
of tactical decomposition matrices of a $2$-$(v,k,\lambda_2;2)$ design.
The system of equations and inequations for coefficients of tactical
decomposition matrices represents necessary conditions
for the existence of designs over finite fields with an
assumed automorphism group.

The Kramer-Mesner method~\cite{KM76} has been adopted and used
for construction of designs over finite fields, see~\cite{MB05}, \cite{BKL05}, \cite{MB13}. In~\cite{KNP11} it was introduced how a tactical decomposition of a $t$-design induced by an action
of a proposed automorphism group can be used for
the enhancement of the Kramer-Mesner method.
The necessary conditions on the existence of designs over finite fields
with an assumed automorphism group introduced in this article can be
implemented in the Kramer-Mesner method for construction of designs over
finite fields, in a manner analog to~\cite{KNP11}.

\begin{definition}
Let $\Psi$ be the set of all
$1$-spaces of a finite vector space $V$
over a finite field $\mathbb{F}$.
Elements of $\Psi$ shall be called \emph{points}.
A \emph{decomposition} of a design
$\B$ over a finite field $\mathbb{F}$ is a partition of the set of points
$\Psi = \Psi_1 \sqcup \cdots \sqcup \Psi_m$
and the set $\B = \B_1 \sqcup \cdots \sqcup \B_n$.
We say that a decomposition is \emph{tactical} if there
exist nonnegative integers $\rho_{ij}$,
$\kappa_{ij}$, $i = 1,\ldots,m$, $j = 1,\ldots,n$,
such that 
\begin{enumerate}
\item every point in $\Psi_i$ is contained
in $\rho_{ij}$ blocks in $\B_j$, 
\item each block in $\B_j$ contains
$\kappa_{ij}$ points in $\Psi_i$.
\end{enumerate}
Matrices $[\rho_{ij}]$ i $[\kappa_{ij}]$ are called
\emph{tactical decomposition matrices}.
\end{definition}

There are two trivial examples of tactical decomposition
of a design. The first example is obtained by putting
$n = m = 1$, and the second
by partitioning sets $\Psi$ and $\B$ into $1$-element subsets.
A nontrivial tactical decomposition can be obtained by
an action of an automorphism group $G \le Aut(\B)$ on a design.

\begin{theorem}
Let $G$ be an automorphism group of a
design over finite field $\B$.
Then the orbits of the set of points $\Psi$ and the orbits of $\B$
form a tactical decomposition.
\end{theorem}

Let $\B$ be a design with parameters $2$-$(v,k,\lambda_2; q)$.
Let
$$
\Psi = \Psi_1 \sqcup \cdots \sqcup \Psi_m,\,\,\,
{\mathcal B} = {\mathcal B}_1 \sqcup \cdots \sqcup
{\mathcal B}_n,
$$
be a tactical decomposition of $\B$.
For $P \in \Psi$ we denote by $\I_P = \{ B \in \B \, | \, P \leq B \}$
the set of all blocks containing $P$. Obviously, $| \I_P | = \lambda_1$
and
\begin{eqnarray*}
\rho_{ij} & = & |\I_P \cap \B_j|, \enspace P \in \Psi_i, \\
\kappa_{ij} & = & | {B \brack 1} \cap \Psi_i|, \enspace B \in \B_j,
\end{eqnarray*}
where ${B \brack 1}$ is the set of all $1$-spaces of $B$.
Coefficients $\rho_{ij}$ and $\kappa_{ij}$
are not dependant on the choice of $P \in \Psi_i$ and of $B \in \B_j$
if and only if the decomposition is tactical.
It is easy to show that
\begin{gather}\label{eqfirst}
\begin{array}{l}
\displaystyle{\sum_{i=1}^m \kappa_{ij} = {k \brack 1}_q,} \\
\displaystyle{\sum_{j=1}^n \rho_{ij} = \lambda_1.} \\
\end{array}
\end{gather}
Double-counting of the set $\{ (P, B) \in \Psi_i \times \B_j \, : \, P \leq B\}$
yields
\begin{equation}
|\Psi_i| \cdot \rho_{ij} = |\B_j| \cdot \kappa_{ij}.
\end{equation}
Now, fix a point $P \in \Psi_{l}$. Double-counting of the set
$$
\{ (Q, B) \in \Psi_r \times \B \, : \, P, Q \leq B \}
$$
yields
\begin{equation}\label{start2}
\sum_{j=1}^n \rho_{l j}\kappa_{r j}  = \sum_{Q \in \Psi_r} |\I_P \cap \I_Q|.
\end{equation}
It is easy to compute the right-hand side of the previous expression.
Obviously, $\I_P \cap \I_Q = \{ B \in \B \, : \, \langle P,Q \rangle \leq B \}$
and so
$$
|\I_P \cap \I_Q| =
\left\{
\begin{array}{c l}
\lambda_1, & P = Q, \\
\lambda_2, & P \neq Q. \\
\end{array}
\right.
$$
Thus, we have obtained a system of equations for the coefficients
of tactical decomposition matrices.
\begin{theorem}
Assume $\B$ is a $2$-$(v,k,\lambda_2;q)$ design
with a tactical decomposition 
$$
\Psi = \Psi_1 \sqcup \cdots \sqcup \Psi_m,\,\,\,
{\mathcal B} = {\mathcal B}_1 \sqcup \cdots \sqcup {\mathcal B}_n.
$$
Let $[\rho_{ij}]$ and $[\kappa_{ij}]$ be the associated tactical decomposition matrices.
Then
\begin{equation}\label{eq2}
\sum_{j = 1}^{n} \rho_{l j} \kappa_{r j} = \left\{
\begin{array}{cl}
\lambda_2 \cdot |\Psi_r|, & l \neq r, \\
\lambda_1 + \lambda_2 \cdot (|\Psi_r|-1), & l = r.
\end{array} \right.
\end{equation}
\end{theorem}

\section{Improvements for the binary field}

Assume now that $\B$ is a design over the binary field $\Fd$
with parameters $2$-$(v,k,\lambda_2;2)$
and with an automorphism group $G \le GL_v(2)$.
Then the orbits of $\Psi$ and the orbits of $\B$ 
form a tactical decomposition
$$
\Psi = \Psi_1 \sqcup \cdots \sqcup \Psi_m,\,\,\,
{\mathcal B} = {\mathcal B}_1 \sqcup \cdots \sqcup
{\mathcal B}_n.
$$
Fix a point $P \in \Psi_{l}$. Double-counting of the set
$$
\{ (R, S, B) \in \Psi_r \times \Psi_{s} \times \B \, : \, P, R, S \leq B \}
$$
yields
\begin{equation}\label{qeq3}
\sum_{j=1}^m \rho_{lj}\kappa_{rj}\kappa_{sj} =
\sum_{R \in \Psi_r} \sum_{S \in \Psi_{s}} | \I_P \cap \I_{R} \cap \I_{S} |.
\end{equation}
Let $R \in \Psi_r$, $S \in \Psi_{s}$. Obviously
$$
\I_P \cap \I_{R} \cap \I_{S} = \{  B \in \B \, : \, \langle P, R, S \rangle \leq B \} =: \I_{P R S}.
$$
By the dimension formula, $1 \le \dim \langle P, R, S \rangle \le 3$.
For $R \in \Psi_r$, let $\Psi_{s}^{i}(R)$ be the set of all $Q \in \Psi_{s}$
such that $\dim \langle P, R, Q \rangle = i$, $i = 1,2,3$,
$$
\Psi_{s}^{i}(R) := \{ Q \in \Psi_{s} \, : \, \dim\langle P, R, Q \rangle = i\}.
$$
Sets $\Psi_{s}^{i}(R)$ are pairwise disjoint and form a partition of $\Psi_{s}$,
$$
\Psi_{s} = \Psi_{s}^{1}(R) \sqcup \Psi_{s}^{2}(R) \sqcup \Psi_{s}^{3}(R).
$$
Let
$$
\phi_{r s}^i = \sum_{R \in \Psi_r} \sum_{S \in \Psi_{s}^i(R)} | \I_{P R S} |, \,\,\, i= 1,2,3.
$$
Then
\begin{equation}\label{sumphi}
\sum_{j=1}^n \rho_{l j} \kappa_{r j} \kappa_{s j} =
\phi_{r s}^1 + \phi_{r s}^2 + \phi_{r s}^3.
\end{equation}
In the continuation, we compute $\phi_{r s}^1$, $\phi_{r s}^2$,
and obtain an upper bound for $\phi_{r s}^3$.
It is easy to see that
$$
| \I_{P R S} | =
\left\{
\begin{array}{cl}
\lambda_1, & S \in \Psi_{s}^{1}(R), \\
\lambda_2, & S \in \Psi_{s}^{2}(R). \\
\end{array}
\right.
$$
For $S \in \Psi_{s}^{3}(R)$ we can obtain only an upper bound,
\begin{equation}\label{qleql}
| \I_{P R S} | \le \min \{ \lambda_2, {v-3 \brack k-3}_q \} =: \varphi.
\end{equation}
Consequently, we can obtain only an upper bound
for the right-hand side of (\ref{sumphi}).

We denote the $3$ points of a $2$-space $\langle P, R \rangle$ of $\Fd^v$
by $P$, $R$ and $P+R$.
Let $\Sm_{r s}(P) \subseteq \Psi_r$
be the set of all points $R \in \Psi_r$, $P \neq R$, such that $P+R \in \Psi_{s}$,
$$
\Sm_{r s}(P) := \{ R \in \Psi_r \setminus \{P \} \, : \, P+R \in \Psi_{s} \}.
$$
Tactical decomposition of $\B$ is group-induced. Hence, the
cardinality of $\Sm_{r s}(P)$ is not dependant of the choice of $P \in \Psi_l$,
i.e. $|\Sm_{r s}(P)| = |\Sm_{r s}(P')|, \forall P' \in \Psi_l$.
We shall write $\sigma_{l r s} := |\Sm_{r s}(P)|$.
The cardinality of $\Psi_{s}^{i}(R)$, $i=1,2,3$,
varies depending on whether $R = P$, $R \in \Sm_{r s}(P)$ or otherwise.

\begin{lemma}\label{qlem}
Assume $\B$ is a $2$-$(v,k,\lambda_2;2)$ design
with an automorphism group $G$ and a $G$-induced
tactical decomposition 
$$
\Psi = \Psi_1 \sqcup \cdots \sqcup \Psi_m, \,\,\,
{\mathcal B} = {\mathcal B}_1 \sqcup \cdots \sqcup {\mathcal B}_n.
$$
Let $P \in \Psi_l$ and $R \in \Psi_r$. Then
$$
|\Psi_{s}^{1}(R)| =
\left\{
\begin{array}{cl}
1, & l = r = s \mbox{ i } R = P, \\
0, & \mbox{otherwise }.  \\
\end{array}
\right.
$$
Furthermore:
\begin{enumerate}
\item For $l \neq r \neq s \neq l$ holds
$$
\begin{array}{l}
|\Psi_{s}^{2}(R)| =
\left\{
\begin{array}{cl}
1, & R \in \Sm_{r s}(P), \\
0, & \mbox{otherwise}, \\
\end{array}
\right.  \\
\\
|\Psi_{s}^{3}(R)| =
\left\{
\begin{array}{cl}
|\Psi_{s}|-1, & R \in \Sm_{r s}(P), \\
|\Psi_{s}|, & \mbox{otherwise}. \\
\end{array}
\right. \\
\end{array}
$$
\item For $l = r = s$ holds
$$
\begin{array}{l}
|\Psi_{s}^{2}(R)| =
\left\{
\begin{array}{cl}
3, & R \in \Sm_{r s}(P), \\
|\Psi_{s}|-1, & R = P, \\
2, & \mbox{otherwise}, \\
\end{array}
\right. \\
\\
|\Psi_{s}^{3}(R)| =
\left\{
\begin{array}{cl}
|\Psi_{s}|-3, & R \in \Sm_{r s}(P), \\
0, & R = P, \\
|\Psi_{s}|-2, & \mbox{otherwise}. \\
\end{array}
\right. \\
\end{array}
$$
\item For $l = r \neq s$ holds
$$
\begin{array}{l}
|\Psi_{s}^{2}(R)| =
\left\{
\begin{array}{cl}
1, & R \in \Sm_{r s}(P), \\
|\Psi_{s}|, & R = P, \\
0, & \mbox{otherwise}, \\
\end{array}
\right. \\
\\
|\Psi_{s}^{3}(R)| =
\left\{
\begin{array}{cl}
|\Psi_{s}|-1, & R \in \Sm_{r s}(P), \\
0, & R = P, \\
|\Psi_{s}|, & \mbox{otherwise}. \\
\end{array}
\right. \\
\\
\end{array}
$$
\item For $l \neq r = s$ holds
$$
\begin{array}{l}
|\Psi_{s}^{2}(R)| =
\left\{
\begin{array}{cl}
2, & R \in \Sm_{r s}(P), \\
1, & \mbox{otherwise}, \\
\end{array}
\right. \\
\\
|\Psi_{s}^{3}(R)| =
\left\{
\begin{array}{cl}
|\Psi_{s}|-2, & R \in \Sm_{r s}(P), \\
|\Psi_{s}|-1, & \mbox{otherwise}. \\
\end{array}
\right. \\
\end{array}
$$
\end{enumerate}
\end{lemma}
\begin{proof}
It is easy to see that
$$
\Psi_{s}^1(R) =
\left\{
\begin{array}{cl}
\{ P \}, & l = r = s \mbox{ and } R = P, \\
\emptyset, & \mbox{otherwise}.
\end{array}
\right.
$$
We shall now determine $\Psi_{s}^2(R)$ in each of the four cases.
Then,
$$
\Psi_{s}^3(R) = \Psi_{s} \setminus (\Psi_{s}^1(R) \sqcup \Psi_{s}^2(R)).
$$
Note that for $R \in \Psi_r$ holds
$$
\Psi_{s}^2(R) \subseteq \{P, R, P + R\}.
$$
Let $l \neq r \neq s \neq l$. Then $P, R \not\in \Psi_{s}$, hence
$$
\Psi_{s}^2(R) =
\left\{
\begin{array}{cl}
\{P+R\}, & R \in \Sm_{r s}(P), \\
\emptyset, & R \not\in \Sm_{r s}(P). \\
\end{array}
\right.
$$
Let $l = r = s$. Then $P, R \in \Psi_l$ and
$$
\Psi_{s}^2(R) =
\left\{
\begin{array}{cl}
\{ P, R, P + R \}, & R \in \Sm_{r s}(P), \\
\Psi_{s} \setminus \{P\},  & R = P, \\
\{P, R\}, & \mbox{otherwise}. \\
\end{array}
\right.
$$
Let $l = r \neq s$. Then
$$
\Psi_{s}^2(R) =
\left\{
\begin{array}{cl}
\{P+R\}, & R \in \Sm_{r s}(P), \\
\Psi_{s}, & R = P, \\
\emptyset, & \mbox{otherwise}. \\
\end{array}
\right.
$$
Let $l \neq r = s$. Then
$$
\Psi_{s}^2(R) =
\left\{
\begin{array}{cl}
\{R, P+R \}, & R \in \Sm_{r s}(P), \\
\{R\}, & \mbox{otherwise}. \\
\end{array}
\right.
$$
\end{proof}
The following theorem gives additional necessary conditions
on the existence of a $2$-$(v,k,\lambda_2;2)$ design
with an assumed automorphism group.
\begin{theorem}\label{tmqeq5}
Assume $\B$ is a $2$-$(v,k,\lambda_2;2)$ design
with an automorphism group $G$ and a $G$-induced
tactical decomposition 
$$
\Psi = \Psi_1 \sqcup \cdots \sqcup \Psi_m, \,\,\,
{\mathcal B} = {\mathcal B}_1 \sqcup \cdots \sqcup {\mathcal B}_n.
$$
Let $[\rho_{ij}]$ and $[\kappa_{ij}]$ be the associated tactical decomposition matrices.
Then
\begin{equation*}\label{qeq22}
\begin{array}{l}
\displaystyle{\sum_{j = 1}^{n} \rho_{lj} \kappa_{rj} \kappa_{sj} =} \\
= \left\{
\begin{array}{ll}
 \sigma_{l r s} \cdot \lambda_2 + \phi_{r s}^3, & l \neq r \neq s \neq l, \\
\lambda_1 + (3 |\Psi_l| + \sigma_{l r s}-3)\lambda_2 + \phi_{r s}^3, & l = r = s, \\
(|\Psi_{s}| + \sigma_{l r s})\lambda_2 + \phi_{r s}^3,
& l = r \neq s \mbox{ or } l \neq r = s , \\
\end{array} \right. \\
\end{array}
\end{equation*}
and
$$
\phi_{r s}^3 \le \left\{
\begin{array}{ll}
(|\Psi_r| \cdot |\Psi_{s}| - \sigma_{l r s}) \varphi, &  l \neq r \neq s \neq l, \\
(|\Psi_l|^2 - 3|\Psi_l|- \sigma_{l r s} + 2) \varphi, & l = r = s, \\
(|\Psi_r| \cdot |\Psi_{s}| - |\Psi_{s}| - \sigma_{l r s}) \varphi, & l = r \neq s \mbox{ or } l \neq r = s, \\
\end{array}
\right.
$$
where $\varphi := \mbox{min}\{\lambda_2, {v-3 \brack k-3}_q \}$.
\end{theorem}
\begin{proof}
Fix a point $P \in \Psi_l$. Then it holds
\begin{equation*}
\sum_{j = 1}^{n} \rho_{lj} \kappa_{rj} \kappa_{sj} = \phi_{r s}^1 + \phi_{r s}^2 + \phi_{r s}^3.
\end{equation*}
Applying Lemma \ref{qlem}, it is easy to compute $\phi_{r s}^1$ and $\phi_{r s}^2$,
and obtain an upper bound for $\phi_{r s}^3$.\\
Let $l \neq r \neq s \neq l$.
Then $\Psi_r = \Sm_{r s}(P) \sqcup \overline{\Sm}_{r s}(P)$,

$$
\phi_{r s}^2 = \sum_{R \in \Psi_r} \sum_{S \in \Psi_{s}^2(R)} \lambda_2
= \sum_{R \in \Sm_{r s}(P)} \sum_{S \in \Psi_{s}^2(R)} \lambda_2 =
 \sigma_{l r s} \cdot \lambda_2,
$$
while
$$
\begin{array}{ll}
\phi_{r s}^3 & = \displaystyle{\sum_{R \in \Psi_r} \sum_{S \in \Psi_{s}^3(R)} |\I_{P R S} |} \\
 & = \displaystyle{\sum_{R \in \Sm_{r s}(P)} \sum_{S \in \Psi_{s}^3(R)} |\I_{P R S} | +
\sum_{R \in \overline{\Sm}_{r s}(P)} \sum_{S \in \Psi_{s}^3(R)} |\I_{P R S} |} \\
& \leq \sigma_{l r s} (|\Psi_{s}|-1) \varphi + (|\Psi_r|-\sigma_{l r s})|\Psi_{s}| \varphi \\
 & = (|\Psi_r| \cdot |\Psi_{s}| - \sigma_{l r s}) \varphi .\\
 \end{array}
$$
Let $l = r = s$. Then $\Psi_r = \Sm_{r s}(P) \sqcup \{P \} \sqcup \overline{\Sm}_{r s}(P)$,
$$
\begin{array}{ll}
\phi_{r s}^2 & =
\displaystyle{\sum_{R \in \Sm_{r s}(P)} \sum_{S \in \Psi_{s}^2(R)} \lambda_2 +
\sum_{R \in \overline{\Sm}_{r s}(P)} \sum_{S \in \Psi_{s}^2(R)} \lambda_2 +
\sum_{S \in \Psi_{s}^2(P)} \lambda_2} \\
 & = 3 \sigma_{l r s} \lambda_2 + 2(|\Psi_{l}|-\sigma_{l r s}-1) \lambda_2
+ (|\Psi_{l}|-1) \lambda_2 \\
 & = (3|\Psi_l| + \sigma_{l r s} - 3)\lambda_2, \\
\end{array}
$$
while
$$
\begin{array}{ll}
\phi_{r s}^3 & =
\displaystyle{\sum_{R \in \Sm_{r s}(P)} \sum_{S \in \Psi_{s}^3(R)} |\I_{P R S} | +
\sum_{R \in \overline{\Sm}_{r s}(P)} \sum_{S \in \Psi_{s}^3(R)} |\I_{P R S} |}\\
 & \leq \sigma_{l r s} (|\Psi_{s}|-3) \varphi + (|\Psi_r|-\sigma_{l r s}-1)(|\Psi_{s}|-2) \varphi \\
 & = (|\Psi_l|^2 - 3|\Psi_l|- \sigma_{l r s} + 2) \varphi. \\
\end{array}
$$
Let $l = r \neq s$. Then $\Psi_l = \Sm_{r s}(P) \sqcup \{ P \} \sqcup \overline{\Sm}_{r s}(P)$,
$$
\begin{array}{ll}
\phi_{r s}^2 & =
\displaystyle{\sum_{R \in \Sm_{r s}(P)} \sum_{S \in \Psi_{s}^2(R)} \lambda_2 +
\sum_{S \in \Psi_{s}^2(P)} \lambda_2} \\
& = (\sigma_{l r s} + |\Psi_{s}|)\lambda_2,
\end{array}
$$
while
$$
\begin{array}{ll}
\phi_{r s}^3 & = \displaystyle{\sum_{R \in \Sm_{r s}(P)} \sum_{S \in \Psi_{s}^3(R)} |\I_{P R S} | +
\sum_{R \in \overline{\Sm}_{r s}(P)} \sum_{S \in \Psi_{s}^3(R)} |\I_{P R S} |} \\
 & \leq (|\Psi_{s}|-1)\sigma_{l r s} \varphi + (|\Psi_r| - \sigma_{l r s} - 1)|\Psi_{s}| \varphi \\
 & = ( |\Psi_r| \cdot |\Psi_{s}| - |\Psi_{s}| - \sigma_{l r s}) \varphi . \\
\end{array}
$$
Let $l \neq r = s$. Then $\Psi_r = \Sm_{r s}(P) \sqcup \overline{\Sm}_{r s}(P)$,
$$
\begin{array}{ll}
\phi_{r s}^2 & =
\displaystyle{\sum_{R \in \Sm_{r s}(P)} \sum_{S \in \Psi_{s}^2(R)} \lambda_2 +
\sum_{R \in \overline{\Sm}_{r s}(P)} \sum_{S \in \Psi_{s}^2(R)} \lambda_2} \\
 & = 2 \sigma_{l r s} \lambda_2 + (|\Psi_{s}| - \sigma_{l r s}) \lambda_2 \\
 & = (|\Psi_{s}| + \sigma_{l r s}) \lambda_2, \\
\end{array}
$$
while
$$
\begin{array}{ll}
\phi_{r s}^3 & = \displaystyle{\sum_{R \in \Sm_{r s}(P)} \sum_{S \in \Psi_{s}^3(R)} |\I_{P R S} | +
\sum_{R \in \overline{\Sm}_{r s}(P)} \sum_{S \in \Psi_{s}^3(R)} |\I_{P R S} |} \\
 & \leq \sigma_{l r s} (|\Psi_{s}|-2) \varphi + (|\Psi_r| -
 \sigma_{l r s})(|\Psi_{s}|-1) \varphi \\
 & = (|\Psi_r| \cdot |\Psi_{s}| - |\Psi_r| - \sigma_{l r s}) \varphi. \\
\end{array}
$$
\end{proof}

Note that the application of the equality
$$
\kappa_{ij} = \frac{|\Psi_i|}{ |\B_j|} \rho_{ij}
$$
on the left-hand side of the expression (\ref{sumphi})
eliminates the coefficients $\kappa_{ij}$ from (\ref{sumphi}) 
and yields a system of inequations for the
coefficients $\rho_{ij}$
$$
\sum_{j = 1}^{n} \rho_{lj} \kappa_{rj} \kappa_{sj} =
\sum_{j=1}^n \frac{|\Psi_r| \cdot |\Psi_{s}|}{ |\B_j|^2}
\rho_{l j} \rho_{r j} \rho_{s j}.
$$
An analog relation is valid for the coefficients $\kappa_{ij}$ as well.

\section{Examples for some cyclic groups}

We shall now illustrate our results on the
example of a design $\B$ with parameters $2$-$(4,3,3;2)$.
Let $G = \langle \Phi \rangle \le GL_2(4)$,
$$
\Phi =
\left[
\begin{array}{cccc}
0 & 0 & 0 & 1\\
0 & 0 & 1 & 0\\
0 & 1 & 1 & 0\\
1 & 0 & 0 & 1\\
\end{array}
\right].
$$
Group $G$ is the cyclic group of order $3$. Assume $G$ is an automorphism group
of $\B$. Then
$$
\Psi = \Psi_1 \sqcup \cdots \sqcup \Psi_5,
$$
with respective orbit representatives
$\langle [  1,  0,  0,  0 ]\rangle$,
$\langle [  1,  0,  1,  0 ]\rangle$,
$\langle [  1,  0,  1,  1 ]\rangle$,
$\langle [  1,  1,  0,  0 ]\rangle$,
$\langle [  0,  1,  0,  0 ]\rangle$.
All orbits are of length $3$. The orbits of $\B$ are currently 
unknown to us, but it is obvious that these orbits are 
of length $3$. In addition, the orbits of $\Psi$
and the orbits of $\B$ form a tactical decomposition
$$
\Psi = \Psi_1 \sqcup \cdots \sqcup \Psi_5, \,\,\,
\B = \B_1 \sqcup \cdots \sqcup \B_5.
$$
Coefficients of a corresponding tactical decomposition matrix $[\rho_{ij}]$
must satisfy the equations (\ref{eqfirst}) and (\ref{eq2}),
\begin{gather*}
\sum_{j=1}^5 \rho_{ij} = 7, \,\,\, i = 1,\ldots,5, \\
\sum_{i=1}^5 \rho_{ij} = 7, \,\,\, j = 1,\ldots,5, \\
\sum_{j=1}^5 \rho_{ij}^2 = 13, \,\,\, i = 1,\ldots,5, \\
\sum_{j=1}^5 \rho_{r j}\rho_{s j} = 9, \,\,\, r \neq s.
\end{gather*}
There are two matrices, up to a rearrangement of
rows and columns, with coefficients that satisfy the above mentioned equations:
$$
\begin{array}{c c c}
\left[
\begin{array}{ccccc}
3 & 1 & 1 & 1 & 1 \\
1 & 3 & 1 & 1 & 1 \\
1 & 1 & 3 & 1 & 1 \\
1 & 1 & 1 & 3 & 1 \\
1 & 1 & 1 & 1 & 3 \\
\end{array}
\right]
&
,\,\,\,
&
\left[
\begin{array}{ccccc}
3 & 1 & 1 & 1 & 1 \\
1 & 2 & 2 & 2 & 0 \\
1 & 2 & 2 & 0 & 2 \\
1 & 2 & 0 & 2 & 2 \\
1 & 0 & 2 & 2 & 2 \\
\end{array}
\right].
\end{array}
$$
Note that the coefficients of the tactical decomposition matrices of a block design
with corresponding parameters $2$-$(15,7,3)$ also
necessarily satisfy this system of equations.

Furthermore, by the Theorem \ref{tmqeq5}, the coefficients
$\rho_{ij}$ satisfy an additional system of inequations.
First we determine the values $\sigma_{l r s}$,
$$
\sigma_{l r s} =
\left\{
\begin{array}{c l}
1, & l \neq r \neq s \neq l, \\
2, & l = r = s, \\
0, & l = r \neq s \mbox{ or } l \neq r = s.\\
\end{array}
\right.
$$
Coefficients $\rho_{ij}$ necessarily satisfy these inequations:
\begin{equation}\label{eqfail}
31 \, \le \, \sum_{j=1}^5 \rho_{l j}^3 \, \le \, 31, \,\,\, l = 1,\ldots,5,
\end{equation}
\begin{equation}
3 \, \le \, \sum_{j=1}^5 \rho_{l j} \rho_{r j} \rho_{s j} \, \le \, 11,\,\,\,
l \neq r \neq s \neq l,
\end{equation}
\begin{equation}
9 \, \le \, \sum_{j=1}^5 \rho_{l j}^2 \rho_{r j} \, \le \, 15,\,\,\,
l \neq r.
\end{equation}
Only one of the two constructed matrices satisfies these additional
constraints. For the matrix
$$
\left[
\begin{array}{ccccc}
3 & 1 & 1 & 1 & 1 \\
1 & 2 & 2 & 2 & 0 \\
1 & 2 & 2 & 0 & 2 \\
1 & 2 & 0 & 2 & 2 \\
1 & 0 & 2 & 2 & 2 \\
\end{array}
\right],
$$
it holds that
$$
\sum_{j=1}^5 \rho_{2 j}^3 = 25,
$$
a contradiction with inequation (\ref{eqfail}).
Hence, the associated tactical decomposition matrix $[\rho_{ij}]$
of a design with parameters
$2$-$(4,3,3;2)$ and automorphism group $G$, equals to the matrix
$$
\left[
\begin{array}{ccccc}
3 & 1 & 1 & 1 & 1 \\
1 & 3 & 1 & 1 & 1 \\
1 & 1 & 3 & 1 & 1 \\
1 & 1 & 1 & 3 & 1 \\
1 & 1 & 1 & 1 & 3 \\
\end{array}
\right],
$$
up to a rearrangement of rows and columns.
There is a unique $2$-$(4,3,3;2)$ design. It can be obtained
by taking all the hyperplanes of the projective space $\PG(3,2)$.
In general, the coefficients of the tactical decomposition matrices of a block design
with the corresponding parameters do not necessarily satisfy the system
of inequations from Theorem \ref{tmqeq5}. Namely, there are block designs
with parameters $2$-$(15,7,3)$ for each of the two constructed tactical
decomposition matrices. For the computation of the matrices we used
the program \verb"orbmat5qd" made by V. Kr\v{c}adinac \cite{KNP11}, application GAP \cite{GAP}
and our own programs.

Hereafter, we give another example.
Consider now a design $\B$ with parameters $2$-$(6,3,6;2)$. Let  $\Phi \in GL_6(2)$,
$$
\Phi =
\left[
\begin{array}{cccccc}
0 & 1 & 0 & 0 & 0 & 0 \\
0 & 0 & 1 & 0 & 0 & 0 \\
0 & 0 & 0 & 1 & 0 & 0 \\
0 & 0 & 0 & 0 & 1 & 0 \\
0 & 0 & 0 & 0 & 0 & 1 \\
1 & 0 & 0 & 0 & 1 & 1 \\
\end{array}
\right] .
$$
Then $G = \langle \Phi \rangle \le GL_6(2)$ is the cyclic group of order $31$.
Assume that $G$ is an automorphism group of $\B$. Then
$$
\Psi = \Psi_1 \sqcup \Psi_2 \sqcup \Psi_3,
$$
with respective orbit representatives
$\langle[  1,  1,  1,  1,  0,  1 ]\rangle$,
$\langle[  1,  0,  0,  0,  0,  0 ]\rangle$,
$\langle[  1,  0,  0,  0,  0,  1 ]\rangle$.
Moreover, $|\Psi_1|=1$, $|\Psi_2|=|\Psi_3| = 31$.
Furthermore, all orbits of $\B$, currently unknown to us, are necessarily of length $31$,
and the orbits of $\Psi$ and the orbits of $\B$ form a tactical decomposition
$$
\Psi = \Psi_1 \sqcup \Psi_2 \sqcup \Psi_3,\,\,\,
\B = \B_1 \sqcup \cdots \sqcup \B_{18}.
$$
In addition,
$$
\begin{array}{l}
\sigma_{1 r s} =
\left\{
\begin{array}{c l}
31, &  1 \neq r \neq s \neq 1, \\
0, & \mbox{otherwise}, \\
\end{array}
\right.
\\
\\
\sigma_{2 r s} =
\left\{
\begin{array}{c l}
30, & r \neq s,\, r, s = 2, 3, \\
1, & 2 \neq r \neq s \neq 2, \\
0, & \mbox{otherwise}, \\
\end{array}
\right.
\\
\\
\sigma_{3 r s} =
\left\{
\begin{array}{c l}
30, & r = s ,\, r, s = 2, 3, \\
1, & 3 \neq r \neq s \neq 3, \\
0, & \mbox{otherwise}. \\
\end{array}
\right.
\\
\end{array}
$$

We constructed $65$ matrices
satisfying the equations (\ref{eqfirst}) and (\ref{eq2})
for coefficients $\rho_{ij}$
of tactical decomposition matrices. Out of these $65$ matrices,
$3$ do not satisfy the system of inequations from Theorem \ref{tmqeq5}:
$$
\left[
\begin{array}{cccccccccccccccccc}
31 & 31 & 0 & 0 & 0 & 0 & 0 & 0 & 0 & 0 & 0 & 0 & 0 & 0 & 0 & 0 & 0 & 0\\
0 & 6 & 1 & 3 & 3 & 3 & 3 & 3 & 4 & 4 & 4 & 4 & 4 & 4 & 4 & 4 & 4 & 4\\
6 & 0 & 6 & 4 & 4 & 4 & 4 & 4 & 3 & 3 & 3 & 3 & 3 & 3 & 3 & 3 & 3 & 3\\
\end{array}
\right],
$$
$$
\left[
\begin{array}{cccccccccccccccccc}
31 & 31 & 0 & 0 & 0 & 0 & 0 & 0 & 0 & 0 & 0 & 0 & 0 & 0 & 0 & 0 & 0 & 0\\
0 & 6 & 2 & 2 & 2 & 3 & 3 & 4 & 4 & 4 & 4 & 4 & 4 & 4 & 4 & 4 & 4 & 4\\
6 & 0 & 5 & 5 & 5 & 4 & 4 & 3 & 3 & 3 & 3 & 3 & 3 & 3 & 3 & 3 & 3 & 3\\
\end{array}
\right],
$$
$$
\left[
\begin{array}{cccccccccccccccccc}
31 & 31 & 0 & 0 & 0 & 0 & 0 & 0 & 0 & 0 & 0 & 0 & 0 & 0 & 0 & 0 & 0 & 0\\
6 & 0 & 5 & 5 & 4 & 4 & 4 & 4 & 4 & 3 & 3 & 3 & 3 & 3 & 3 & 3 & 3 & 2\\
0 & 6 & 2 & 2 & 3 & 3 & 3 & 3 & 3 & 4 & 4 & 4 & 4 & 4 & 4 & 4 & 4 & 5\\
\end{array}
\right].
$$
Namely, the coefficients of these $3$ matrices
do not satisfy the inequality
$$
186 \, \le \, \sum_{j=1}^{18} \rho_{1j}\rho_{2j}\rho_{3j} \, \le \, 1116.
$$
For each of the remaining $62$ matrices we attempted to construct
a design over finite field with parameters $2$-$(6,3,6;2)$,
automorphism group $G$ and associated tactical decomposition
matrix $M$. For the construction we used a method analog to the one described in~\cite{KNP11}. 
We conclude that such design exists only when $M$ is
$$
\left[
\begin{array}{cccccccccccccccccc}
31 & 31 & 0 & 0 & 0 & 0 & 0 & 0 & 0 & 0 & 0 & 0 & 0 & 0 & 0 & 0 & 0 & 0 \\
3 & 3 & 0 & 0 & 4 & 4 & 4 & 4 & 4 & 4 & 4 & 4 & 4 & 4 & 4 & 4 & 4 & 4 \\
3 & 3 & 7 & 7 & 3 & 3 & 3 & 3 & 3 & 3 & 3 & 3 & 3 & 3 & 3 & 3 & 3 & 3 \\
\end{array}
\right].
$$
In~\cite{MB05} examples of $2$-$(6,3,6;2)$ were constructed. Using the Kramer-Mesner method the author constructed designs with an automorphism group $G = \langle \sigma^7 \rangle$, where $\sigma$ is the Singer cycle, hence, $G$ is the cyclic group of order $9$. The number of constructed designs is not reported. We could construct $330$ designs with given parameters, admitting the action of the cyclic group $G$ of order $31$.

\bibliographystyle{model1-num-names}

\end{document}